\documentclass[11pt]{article}
\usepackage{amsmath, amssymb, amsthm, verbatim,enumerate,bbm}
\usepackage{indentfirst}
\usepackage{amsmath}
\usepackage{amsthm}
\usepackage{amsfonts}
\usepackage{amssymb}
\usepackage{verbatim}
\usepackage{graphicx}
\usepackage[margin=1.0in]{geometry}
\usepackage{url}
\usepackage{enumerate}
\usepackage[pdftex,bookmarks=true]{hyperref}
\usepackage{todonotes}

\title{A nonuniform Littlewood-Offord inequality for all norms}

\author{ Kyle Luh
	\thanks{Department of Mathematics, University of Colorado Boulder. Email: kyle.luh@colorado.edu.} \and David Xiang \thanks {Harvard University. Email: davidxiang@college.harvard.edu.}
}

\date{\today}
\allowdisplaybreaks

\theoremstyle{plain}
\newtheorem{theorem}{Theorem}[section]
\newtheorem{lemma}[theorem]{Lemma}

\newtheorem{proposition}[theorem]{Proposition}

\newtheorem{conjecture}[theorem]{Conjecture}

\newtheorem{definition}[theorem]{Definition}

\newcommand{\R}{\mathbb{R}}
\newcommand{\eps}{\varepsilon}

\renewcommand{\P}{\mathbb P}

\newcommand{\Bv}{\mathbf{v}}
\newcommand{\Bu}{\mathbf{u}}
\newcommand{\Bx}{\mathbf{x}}
\newcommand{\By}{\mathbf{y}}
\newcommand{\Ba}{\mathbf{a}}

\begin{document}
	\maketitle
	
	\begin{abstract}
		Let $\Bv_i$ be vectors in $\R^d$ and $\{\eps_i\}$ be independent Rademacher random variables.  Then the Littlewood-Offord problem entails finding the best upper bound for $\sup_{\Bx \in \R^d} \P(\sum \eps_i \Bv_i = \Bx)$.  
		Generalizing the uniform bounds of Littlewood-Offord, Erd\H{o}s and Kleitman, a recent result of Dzindzalieta and Ju\v{s}kevi\v{c}ius provides a non-uniform bound that is optimal in its dependence on $\|\Bx\|_2$.  In this short note, we provide a simple alternative proof of their result.  Furthermore, our proof demonstrates that the bound applies to any norm on $\R^d$, not just the $\ell_2$ norm.  This resolves a conjecture of Dzindzalieta and Ju\v{s}kevi\v{c}ius. 
	\end{abstract}

	\section{Introduction}
	Let $\{\eps_k\}_{k=1}^n$ be independent Rademacher random variables (i.e. $\P(\eps_k = 1) = 1/2$ and $\P(\eps_k = -1) = 1/2$).  We let $R_n$ denote the sum of these random variables.  In their study of random polynomials, Littlewood and Offord \cite{LO1943roots} encountered the following problem.  What is the best bound on $\P(\sum_{i=1}^n a_i \eps_i = x)$ with $|a_i| \leq 1$.  Littlewood and Offord established that $\max_x \P(\sum a_i \eps_i = x) = O(\log n/ n^{1/2})$ for all $a_i$ such that $|a_i| \leq 1$.  \cite{LO1943roots}.  With a short, insightful argument, Erd\H{o}s \cite{erdos1945LO} established the optimal bound 
	\begin{equation} \label{eq:erdosLO}
	\rho(\Ba) :=\max_x \P\left(\sum_{i=1}^n a_i \eps_i = x \right) \leq \frac{\binom{n}{\lfloor n/2 \rfloor}}{2^n} = O(n^{-1/2}).
	\end{equation}

	The results of Littlewood, Offord and Erd\H{o}s attracted the attention of many researchers and numerous variants of the Littlewood-Offord problem have been proposed and investigated.  Erd\H{o}s and Moser showed that an improved bound held when all the $a_i$ are distinct \cite{erdos1947moser}.  Later, S\'ark\H{o}zy and Szemer\'edi obtained the optimal bound for distinct $a_i$.  Many more results were obtained when considering more complex arithmetic structure of the $a_i$'s \cite{stanley1980distinct, halasz1977combinatorial, rv2008invertibility}.  
	In a different direction, Erd\H{o}s conjectured that a result analogous to \eqref{eq:erdosLO} should hold in in higher dimensions.  This extension was non-trivial and it took two decades before such a result was verified by Kleitman \cite{kleitman1966combinatorial}.
	\begin{theorem} 
		Let $d \in \mathbb{N}$ and $\Bv_i \in \R^d$ with $\|\Bv_i\|_2 \leq1$ and $\Bv_i \neq 0$.  Then,
		\[
		\rho(\Bv_1, \dots, \Bv_n) := \sup_{\Bx \in \R^d} \P\left(\sum_{i=1}^n \eps_i \Bv_i = \Bx \right) \leq \frac{\binom{n}{\lfloor n/2 \rfloor}}{2^n}.
		\] 
	\end{theorem}
	Inspired by the inverse problems of additive combinatorics, Tao and Vu began a line of work to known as \emph{inverse} Littlewood-Offord  theorems which attempt to explain when $\rho(\Ba)$ is large \cite{taovu2009ILO}.  Essentially, they showed that $\rho(\Ba)$ is large only when the entries of $\Ba$ reside in a generalized arithmetic progression.  Many results in this direction followed and culminated in the optimal inverse Littlewood-Offord theorems of Nguyen and Vu \cite{nguyen2011optimal}.  This theory and its variants played an important role in estimating the singularity probability of random matrices (see \cite{taovu2009ILO, rv2008invertibility, ferber2019counting, tikhomirov2020bernoulli} and the references therein). 
	
	In another vein of work, Tiep and Vu \cite{tiep2016nonabelian} obtained a Littlewood-Offord-type inequality in the setting of non-commutative groups and Ju\v{s}kevi\v{c}ius and \v{S}emetulskis obtained optimal bounds for arbitrary groups.  Bandeira, Ferber and Kwan proposed a new perspective and investigated a resilience version of the Littlewood-Offord problem, namely the number of coefficients in $\Ba$ that an adversary can change to force $\rho(\Ba)$ to be large \cite{bandeira2017resilience}. 
	
	Recently, Dzindzalieta and Ju\v{s}kevi\v{c}ius established a non-uniform Littlewood-Offord inequality in all dimensions.  The bound is non-uniform in that it incorporates information about the vector $\Bx$.
	\begin{theorem} \label{thm:l2}
		Let $\Bv_i \in \R^d$ with $\|\Bv_i\|_2 \leq 1$ and $\Bv_i \neq 0$ for all $i \in [n]$.  Then,
		\[
		\P(\sum_{i=1}^n \eps_i \Bv_i = \Bx) \leq \P(R_n = k + \delta_{n,k}) = \frac{\binom{n}{\lceil \frac{n+k}{2} \rceil}}{2^n}.
		\]
		where $k = \lceil \|\Bx\|_2 \rceil$ and 
		\[
		\delta_{n,k} = 
		\begin{cases}
		1 & n+k \text{ is even} \\
		0 & \text{otherwise}. 
		\end{cases}
		\]
	\end{theorem}
	This result is optimal in $n$ and $\|\Bx\|_2$ as can be seen by setting $\Bv_i = (\frac{\|\Bx\|_2}{k + \delta_{n,k}}, 0 , \dots, 0)$.  In \cite{dzindzalieta2020nonuniform}, it was conjectured that the result should hold for any norm on $\R^d$, not just the $\ell_2$ norm.  
	
	\begin{conjecture} \cite[Conjecture 2]{dzindzalieta2020nonuniform} \label{conj:main}
		Let $\|\cdot \|$ be an arbitrary norm on $\R^d$.  Let $\Bv_i \in \R^d$ be such that $\|\Bv_i\| \leq 1$ and $\Bv_i \neq 0$ for all $i \in [n]$.  Then,
		\[
		\P\left( \sum \eps_i \Bv_i = \Bx \right) \leq \P(R_n = k + \delta_{n,k}).
		\]
	\end{conjecture}
	In \cite{dzindzalieta2020nonuniform}, they used a rotation argument to reduce the multi-dimensional case to the one dimensional case.  However, their rotation only preserves the $\ell_2$ norm and so their argument only applies to this norm.  
	In this short note, we provide an alternate proof of the main result in \cite{dzindzalieta2020nonuniform} and prove Conjecture \ref{conj:main}.
	\begin{theorem} \label{thm:main}
		Let $\|\cdot \|$ be an arbitrary norm on $\R^d$.  Let $\Bv_i \in \R^d$ be such that $\|\Bv_i\| \leq 1$ and $\Bv_i \neq 0$ for all $i \in [n]$.  Then,
		\[
		\P\left( \sum \eps_i \Bv_i = \Bx \right) \leq \P(R_n = k + \delta_{n,k}).
		\]
		where $k = \lceil \|\Bx\| \rceil$ and 
		\[
		\delta_{n,k} = 
		\begin{cases}
		0 & n+k \text{ is even} \\
		1 & \text{otherwise}. 
		\end{cases}
		\]
	\end{theorem}

	\section*{Acknowledgements}
	We thank Victor Reis and Aleksei Kulikov for pointing out an error in our first draft.  We also thank Aleksei for suggesting the perturbation argument at the end of the note.    
	
	\section{Auxiliary Results}
	We will make use of the following one dimensional non-uniform Littlewood-Offord bound.
	\begin{proposition} \cite{dzindzalieta2020nonuniform}\label{prop:1d}
		For non-zero $a_i \in \R$ such that $|a_i|\leq 1$, we have that
		\[
		\P(\sum_{i=1}^n \eps_i a_i = x) \leq \P(R_n = k + \delta_{n,k})
		\]
		where $k = \lceil |x| \rceil$.
	\end{proposition}
	
	We will also utilize the basic theory of dual norms.  
	\begin{definition}
		Let $\|\cdot\|$ be a norm on $\R^d$.  Let $\|\cdot \|_*$ denote the dual norm where for any $\Bu \in \R^d$,
		\[
		\|\Bu\|_* = \sup\{ \langle \Bu, \Bx \rangle: \|\Bx\|\leq 1\}
		\] 
		where $\langle \cdot, \cdot \rangle$ is the standard inner product on $\R^d$.
	\end{definition}
	We then have a basic Cauchy-Schwarz type inequality.  We include the elementary proof for the reader's convenience.  
	
	\begin{lemma} \label{lem:inequality}
		Let $\|\cdot\|$ be a norm on $\R^d$ and $\|\cdot \|_*$ be its dual.  Then, for $\Bx, \By \in \R^d$,
		\[
		|\langle \Bx, \By \rangle| \leq \|\Bx\| \|\By\|_*
		\]
	\end{lemma}
	
	\begin{proof}
		Let $\Bv = \Bx/\|\Bx\|$.  Then we have
		\begin{align*}
		\langle \Bx, \By \rangle &= \|\Bx\| \langle \Bv, \By \rangle \\
		&\leq   \|\Bx\| \|\By\|_*.
		\end{align*}
		To include the absolute value, we apply the same argument to $-\Bx$.  
	\end{proof}
	
	Additionally, we will make use of the standard fact that in finite-dimensional spaces, the double dual norm is the same as the original norm.
	\begin{lemma} (e.g. \cite[Theorem 1.11.9]{megginson2012introduction})\label{lem:doubledual}
		Let $\|\cdot\|$ be a norm on $\R^d$.  Then, for $\Bx \in \R^d$,
		\[
		\|\Bx \| = \|\Bx\|_{**}.
		\] 
	\end{lemma}

	\section{Proof of Theorem \ref{thm:main}}
	\begin{proof}
		For any $\By \in \R^d$, we have that
		\[
		\P\left(\sum_{i=1}^n \eps_i \Bv_i = \Bx \right) \leq \P\left(\left\langle \sum_{i=1}^n \eps_i \Bv_i, \By \right\rangle =  \langle \Bx, \By \rangle\right)
		\]
		
		In particular, if we let $\By = \text{argmax}_{\|\Bu\|_*\leq 1} \langle \Bx, \Bu\rangle$, we can conclude that
		\begin{align*}
		\P\left(\sum_{i=1}^n \eps_i \Bv_i = \Bx \right) &\leq \P\left(\left\langle \sum_{i=1}^n \eps_i \Bv_i, \By \right\rangle =  \langle \Bx, \By\rangle \right) \\
		&= \P\left(\sum_{i=1}^n  \left\langle  \Bv_i, \By \right\rangle \eps_i =  \|\Bx\|_{**} \right) .\\
		\end{align*}
		
		Since $\|\Bv\| \leq 1$ by assumption, Lemma \ref{lem:inequality} implies that
		\[\left|\left\langle  \Bv_i, \By \right\rangle \right | \leq \|\Bv_i\| \|\By\|_*  \leq 1\].  
		
		Therefore, we can apply Proposition \ref{prop:1d}, so
		\[
		\P\left( \sum \eps_i \Bv_i = \Bx \right) \leq \P(R_n = k + \delta_{n,k})
		\]
		where $k = \lceil \|\Bx\|_{**} \rceil = \lceil \|\Bx\| \rceil$.  This final equality follows from Lemma \ref{lem:doubledual}.  In our application of Proposition \ref{prop:1d}, we implicitly assumed that $\langle  \Bv_i, \By \rangle \neq 0$.  To ensure this, we can simply choose a small perturbation of $\By$ such that $\lceil \langle \By, \Bx \rangle \rceil =  \lceil \|\Bx\|_{**} \rceil$.  
	\end{proof}

	\bibliographystyle{abbrv}
	\bibliography{nonuniform}

\begin{thebibliography}{10}

\bibitem{bandeira2017resilience}
A.~S. Bandeira, A.~Ferber, and M.~Kwan.
\newblock Resilience for the {L}ittlewood-{O}fford problem.
\newblock {\em Adv. Math.}, 319:292--312, 2017.

\bibitem{dzindzalieta2020nonuniform}
D.~Dzindzalieta and T.~Ju\v{s}kevi\v{c}ius.
\newblock A non-uniform {L}ittlewood-{O}fford inequality.
\newblock {\em Discrete Math.}, 343(7):111891, 5, 2020.

\bibitem{erdos1945LO}
P.~Erd\"{o}s.
\newblock On a lemma of {L}ittlewood and {O}fford.
\newblock {\em Bull. Amer. Math. Soc.}, 51:898--902, 1945.

\bibitem{erdos1947moser}
P.~Erdos and L.~Moser.
\newblock Elementary {P}roblems and {S}olutions: {S}olutions: {E}736.
\newblock {\em Amer. Math. Monthly}, 54(4):229--230, 1947.

\bibitem{ferber2019counting}
A.~Ferber, V.~Jain, K.~Luh, and W.~Samotij.
\newblock On the counting problem in inverse littlewood--offord theory.
\newblock {\em arXiv preprint arXiv:1904.10425}, 2019.

\bibitem{halasz1977combinatorial}
G.~Hal\'{a}sz.
\newblock Estimates for the concentration function of combinatorial number
  theory and probability.
\newblock {\em Period. Math. Hungar.}, 8(3-4):197--211, 1977.

\bibitem{kleitman1966combinatorial}
D.~J. Kleitman.
\newblock On a combinatorial conjecture of {E}rd{\H o}s.
\newblock {\em Journal of Combinatorial Theory}, 1(2):209--214, 1966.

\bibitem{LO1943roots}
J.~E. Littlewood and A.~C. Offord.
\newblock On the number of real roots of a random algebraic equation. {III}.
\newblock {\em Rec. Math. [Mat. Sbornik] N.S.}, 12(54):277--286, 1943.

\bibitem{megginson2012introduction}
R.~E. Megginson.
\newblock {\em An introduction to Banach space theory}, volume 183.
\newblock Springer Science \& Business Media, 2012.

\bibitem{nguyen2011optimal}
H.~Nguyen and V.~Vu.
\newblock Optimal inverse {L}ittlewood-{O}fford theorems.
\newblock {\em Adv. Math.}, 226(6):5298--5319, 2011.

\bibitem{rv2008invertibility}
M.~Rudelson and R.~Vershynin.
\newblock The {L}ittlewood-{O}fford problem and invertibility of random
  matrices.
\newblock {\em Adv. Math.}, 218(2):600--633, 2008.

\bibitem{stanley1980distinct}
R.~P. Stanley.
\newblock Weyl groups, the hard {L}efschetz theorem, and the {S}perner
  property.
\newblock {\em SIAM J. Algebraic Discrete Methods}, 1(2):168--184, 1980.

\bibitem{taovu2009ILO}
T.~Tao and V.~H. Vu.
\newblock Inverse {L}ittlewood-{O}fford theorems and the condition number of
  random discrete matrices.
\newblock {\em Ann. of Math. (2)}, 169(2):595--632, 2009.

\bibitem{tiep2016nonabelian}
P.~H. Tiep and V.~H. Vu.
\newblock Non-abelian {L}ittlewood-{O}fford inequalities.
\newblock {\em Adv. Math.}, 302:1233--1250, 2016.

\bibitem{tikhomirov2020bernoulli}
K.~Tikhomirov.
\newblock Singularity of random {B}ernoulli matrices.
\newblock {\em Ann. of Math. (2)}, 191(2):593--634, 2020.

\end{thebibliography}
	
\end{document}